\documentclass[12pt,reqno]{amsart}
\usepackage{algorithmic}
\usepackage{fullpage}
\usepackage{amsfonts}
\usepackage{amssymb}
\usepackage{sagetex}
\newtheorem{theorem}{Theorem}[section]
\newtheorem{lemma}[theorem]{Lemma}
\newtheorem{proposition}[theorem]{Proposition}

\newtheorem{remark}[theorem]{Remark}

\renewcommand{\pmod}[1]{{\ifmmode\text{\rm\ (mod~$#1$)}\else\discretionary{}{}{\hbox{ }}\rm(mod~$#1$)\fi}}

\begin{document}

\title[Modular Forms]{Ternary Quadratic Forms And Half-Integral Weight Modular Forms}
\author{Alia Hamieh}
\address{Department of Mathematics \\ University of British Columbia \\ Room
121, 1984 Mathematics Road\\ Vancouver, British Columbia \\ Canada V6T 1Z2}
\email{ahamieh@math.ubc.ca}
\thanks{}

\begin{abstract}
Let $k$ be a positive integer such that $k\equiv3\mod4$, and let $N$ be a positive square-free integer. In this paper, we compute a basis for the two-dimensional subspace $S_{\frac{k}{2}}(\Gamma_{0}(4N),F)$ of half-integral weight modular forms associated, via the Shimura correspondence, to a newform $F\in S_{k-1}(\Gamma_{0}(N))$, which satisfies $L(F,\frac{1}{2})\neq0$. This is accomplished by using a result of Waldspurger, which allows one to produce a basis for the forms that correspond to a given $F$ via local considerations, once a form in the Kohnen space has been determined.\end{abstract}
\maketitle
\thispagestyle{empty}
\setcounter{section}{-1}

\section{Introduction}
Let $k$ be an odd positive integer, and $M$ be a positive integer divisible by $4$. A modular form of half-integral weight $\frac{k}{2}$ for $\Gamma_{0}(M)$ is a holomorphic function $f$ on the upper half-plane which is also holomorphic at the cusps and transforms like the $k$-th power of the theta series $$\displaystyle{\theta(z)=1+2\sum_{n\geq1}e^{2\pi in^{2}z}}$$ under fractional linear transformations of $\Gamma_{0}(M)$. If $f$ vanishes at all the cusps, we say that it is a cusp form and write $f\in S'_{\frac{k}{2}}(\Gamma_{0}(M))$. We also denote by $S'_{\frac{k}{2}}(M,\boldsymbol{\chi})$, the space of cusp forms with central character $\boldsymbol{\chi}$ of conductor dividing $M$. The Kohnen subspace in $S'_{\frac{k}{2}}(M,\boldsymbol{\chi})$ consists of $\displaystyle{f(z)=\sum_{n\geq1}a_{n}e^{2\pi inz}}$ with the Fourier coefficients $a_{n}$ satisfying $$a_{n}=0,\text{ when } \boldsymbol{\chi}_{(2)}(-1)(-1)^{\frac{k-1}{2}}n\equiv2,3 \mod{4},$$ where $\boldsymbol{\chi}_{(2)}$ is the $2$-primary component of $\boldsymbol{\chi}$.

One can define an action of Hecke operators on the space of half-integral weight modular forms. The different Hecke operators commute $(T_{m^{2}}T_{n^{2}}=T_{m^{2}n^{2}})$, and the operator $T_{p^{2v}}$ is a polynomial in $T_{p^{2}}$. However, one can have non-trivial operators $T_{m}$ only for square $m$ or for $(m,M)\ne1$. Hence, for a half-integral weight modular form which is an eigen-form for all Hecke operators, one can only relate its coefficients whose indices differ by a perfect square $($\cite{koblitz}, Proposition $14$$)$. For a more detailed exposition on half-integral weight modular forms the reader is referred to \cite{koblitz}. 

In 1973, Shimura proved a fundamental result which gives a correspondence between modular forms of half-integral weight and modular forms of even integral weight. Let $S_{\frac{3}{2}}(M,\boldsymbol{\chi})$ denote the orthogonal complement $($with respect to the Petersson inner product$)$ of the subspace of $S'_{\frac{3}{2}}(M,\boldsymbol{\chi})$ spanned by the Shimura theta series $$\theta_{\psi,m}(z)=\sum_{n=-\infty}^{\infty}\psi(n)ne^{2\pi in^{2}mz}$$ for all positive integers $m$ and odd primitive Dirichlet characters $\psi$ $($see \cite{sturm}$)$. For an odd integer $k\geq5$, we put for notational convenience $S_{\frac{k}{2}}(M,\boldsymbol{\chi})=S'_{\frac{k}{2}}(M,\boldsymbol{\chi})$. Adapting the notation from \cite{waldspurger}, we let $$S_{k-1}^{\mathrm{new}}(\boldsymbol{\chi}^{2})=\bigcup_{N>0}S_{k-1}^{\mathrm{new}}(N,\boldsymbol{\chi}^{2}),$$ where $S_{k-1}^{\mathrm{new}}(N,\boldsymbol{\chi}^{2})$ is the (finite) subset of newforms in $S_{k-1}(N,\boldsymbol{\chi}^{2})$. If $F\in S_{k-1}^{\mathrm{new}}(\boldsymbol{\chi}^{2})$ is such that $T_{p}F=b_{p}F$ for all $p$, one defines the Shimura lift of $F$ to be the subspace $$S_{\frac{k}{2}}(M,\boldsymbol{\chi},F)=\{f\in S_{\frac{k}{2}}(M,\boldsymbol{\chi}):T_{p^2}f=b_{p}f \text{ }\text{ } \text{for almost all } p\not | M\}.$$ 
Shimura showed that if $f\in S_{\frac{k}{2}}(M,\boldsymbol{\chi})$ is an eigenform for almost all Hecke operators, then there exists a unique $F\in S_{k-1}^{\mathrm{new}}(\boldsymbol{\chi}^{2})$ such that $f\in S_{\frac{k}{2}}(M,\boldsymbol{\chi},F)$. This assignment is at the heart of the Shimura correspondence. The following is a simplified version of Shimura's original theorem. 
\begin{theorem}\label{thm:shimura2}
$($See \cite{shimura}$)$ Let $f \in S_{\frac{k}{2}}(M,\boldsymbol{\chi})$ be a common eigenfunction for all $T_{p^2}$ with $\lambda_{p}$ being the corresponding eigenvalue. Define the sequence of complex numbers $\{b_{n}\}$ by the formal identity $$\sum_{n=1}^{\infty}b_{n}n^{-s} = \prod_{p}\frac{1}{1-\lambda_{p}p^{-s}+\boldsymbol{\chi}(p^2)p^{k-2-2s}}$$
Then $F(z)=\sum_{n=1}^{\infty}b_{n}e^{2\pi inz}$  belongs to $ S_{k-1}(N',\boldsymbol{\chi}^2)$ for some integer $N'$ which is divisible by the conductor of $\chi^{2}$. 
\end{theorem}
Thus, each eigenform of weight $\frac{k}{2}$ is associated to a form of integral weight $k-1$. However, it is not clear  (and in general very far from true) that this correspondence is bijective on the level of eigenforms. Indeed, the principal point for us in this paper is that the number of eigenforms of weight $\frac{k}{2}$ of a given level and character associated to a fixed $F$ is a rather subtle invariant. To clarify the situation, we introduce the following setup. Let $F\in\displaystyle{S_{k-1}^{\text{new}}(\boldsymbol{\chi}^{2})}$ be a newform of even weight $k-1$, level $N$ and character $\boldsymbol{\chi}^{2}$ defined modulo $N$. Let $M$ be a positive integer divisible by $4$, and suppose that $\boldsymbol{\chi}$ is defined modulo $M$. (Here $N$ may or may not divide $M$.) Consider the space $S_{\frac{k}{2}}(M,\boldsymbol{\chi},F)$ defined in a previous paragraph. We are led to the following questions:
\begin{enumerate}
\item {What is the dimension of the space $S_{\frac{k}{2}}(M,\boldsymbol{\chi},F)$?}
\item{If it is non-zero, can we compute a basis for this space?}
\end{enumerate}

It turns out that the answer to these questions is extremely delicate, and relies in a fundamental way on the representation theory of the metaplectic covers of $SL_{2}$  and $GL_{2}$. The answer is rendered complicated by two main factors: firstly, that a certain global theta correspondence is trivial when a certain $L$-function vanishes, and secondly, that there are severe local complications in the representation theory of the metaplectic group. The first causes a natural construction of forms in $S_{\frac{k}{2}}(M,\boldsymbol{\chi},F)$ to vanish, while the second shows that there is no theory of newforms on the half-integral side, and that the space one is trying to construct may have high dimension.

The construction of forms in $S_{\frac{k}{2}}(M,\boldsymbol{\chi},F)$ if non-empty for some $M$ and $\boldsymbol{\chi}$ was taken up by Shintani \cite{shintani}, and solved in general by Flicker \cite{flicker}. It is not at all clear, $a$ priori, which characters $\boldsymbol{\chi}$ and which integers $M$ one has to take in order to obtain a non-zero space.

On the other hand, the questions of finding the dimension and a basis for $S_{\frac{k}{2}}(M,\boldsymbol{\chi},F)$ were taken up by Waldspurger in \cite{waldspurger}. In fact, Waldspurger proved that, under quite general conditions on $F$, $N$ and $\boldsymbol{\chi}$, there exists a basis for $S_{\frac{k}{2}}(M,\boldsymbol{\chi},F)$ such that for every positive integer $n$ the Fourier coefficient $a_{n}(f_{i})$ of a basis element $f_{i}$ is the product of two factors: a product of local terms $c_{i}(n,F)$ each of which is completely determined by the local components of $F$ according to explicit formulae given in \cite{waldspurger}, and a global factor $A_{F}(n)$ whose square is the central critical value of the $L$-function of the newform $F$ twisted with a quadratic character depending on $n$. 

A more detailed discussion of the representation-theoretic subtleties of this circle of questions is given in Section 2 below. The main point of this paper is to answer these questions in the simplest case, when $\boldsymbol{\chi}$ is trivial, and $M = 4N$, with $N$ odd and square-free. More precisely, we will compute a basis for $S_{\frac{k}{2}}(\Gamma_{0}(4N),F)$ for an odd square-free integer $N$ and $k\equiv3\mod{4}$ provided that $L(F,\frac{1}{2})\neq0$ $($the $L$-function is normalized so that the functional equation is with respect to $s\rightarrow 1-s$ rather than $s\rightarrow k-1-s$$)$. In the light of Waldspurger's work, our task is reduced to computing the global factors $A_{F}(n)$.  

The organization of the paper is as follows. In Section 1, we present a construction of a half-integral weight modular form $g$ which belongs to the Kohnen subspace of $S_{\frac{k}{2}}(\Gamma_{0}(4N),F)$ for a given newform $F\in S_{k-1}^{\mathrm{new}}(\Gamma_{0}(N))$. The squares of the Fourier coefficients of this form are essentially proportional to the central critical values of the $L$-function of $F$ twisted with some quadratic characters. In Section 2 we digress briefly upon the representation-theoretic interpretation of the Shimura correspondence as a theta correspondence. In Section 3, we state the full result of Waldspurger. Next, we show that the space $S_{\frac{k}{2}}(\Gamma_{0}(4N),F)$ is in fact two-dimensional and has a distinguished basis $\{f_{1},f_{2}\}$ such that either $f_{1}-f_{2}$ or $f_{2}$ belongs to the Kohnen subspace. We use this to express the global factors $A_{F}(n)$ in terms of the coefficients of $g$ (Theorem \ref{prop}). Finally, we explicitly determine the Fourier coefficients of the two modular forms that generate $S_{\frac{k}{2}}(\Gamma_{0}(4N),F)$, thus, arriving at our main contribution given in Theorem \ref{main}. The last section contains examples to illustrate the calculations carried out in Section $3$.

In conclusion, we remark that the problem of finding a basis for $S_{\frac{k}{2}}(M,\boldsymbol{\chi},F)$ (in a more general setting than the one assumed in our present work) can undoubtedly be solved by representation-theoretic techniques, and by generalizing the framework sketched in Section 2 of the present article. We hope to take this up in a future work.

\section{A Modular Form in $S_{\frac{k}{2}}(\Gamma_{0}(4N),F)$}

We first recall the definition of the $L$-function associated to a newform $F\in S_{k-1}^{\mathrm{new}}(\Gamma_{0}(N))$. Following Waldspurger in \cite{waldspurger}, we set $$L(F,s)=2(2\pi)^{1-s-k/2}N^{\frac{s-1+k/2}{2}}\Gamma(s-1+k/2)\sum_{n=1}^{\infty}b_{n}(F)n^{1-s-k/2}\text{ }\text{ }\text{ }\text{ }\operatorname{Re}s>\frac{3}{2}.$$ Then $L(F,s)$ extends to an entire function which satisfies a functional equation with respect to $s\rightarrow1-s$. Moreover, since we assume at the outset that $F$ is a newform with the trivial character, $F$ is also an eigenform for the Atkin-Lehner Operator $W_{N}$, so that $W_{N}F=w_{N}F$ for some $w_{N}\in\mathbb{C}$. Hence, the functional equation for $L(F,s)$ takes on the form: $$L(F,s)=i^{k-1}w_{N}L(F,1-s).$$
It is known that $W_{N}$ is an involution when $k-1$ is even, so $w_{N}=\pm1$. In particular, if $k\equiv3\mod4$ and $L(F,\frac{1}{2})\neq0$, then the root number in the functional equation is $i^{k-1}w_{N}=-w_{N}=1$. Otherwise, $L(F,\frac{1}{2})$ would vanish trivially.

Let $F$ be a newform in $S_{2}^{\mathrm{new}}(\Gamma_{0}(N))$ with an odd square-free level $N$ such that $L(F,\frac{1}{2})\neq0$. For each prime divisor $p$ of $N$ we denote by $w_{p}$ the eigenvalue of the Atkin-Lehner involution $W_{p}$. Since $w_{N}=\prod w_{p}$ and $w_{N}=-1$ by the above discussion, the set $S=\{p|N: w_{p}=-1\}$ has an odd cardinality. Notice that if we consider a newform of higher even weight $k-1$, then $|S|$ is odd if we also assume that $k\equiv3\mod{4}$. 

Consider the definite quaternion algebra $B$ over $\mathbb{Q}$ ramified at $S$ and $\infty$. Let $O$ be an order in $B$ such that $O_{q}$ is maximal in $B_{q}$ for all $q\in S$ and is of index $p$ in a maximal order for the remaining primes $p$ dividing $N$. Such an order is called an Eichler order of square-free level $N$ in $B$. The relation between orders of level $N$ and modular forms on $\Gamma_{0}(N)$ will be clear in what follows. 
   
The  Brandt module which we shall denote by $X(\mathbb{R})$ is the free abelian group on the left ideal classes of a  quaternion order of level $N$ along with a natural Hecke action determined by the Brandt matrices. Pizer proved in \cite{Pizer} that there exists a Hecke algebra isomorphism between the Brandt module and a subspace of modular forms containing all the newforms of level $N$ (\cite{Pizer}, Corollary 2.29 and Remark 2.30).  In the representation theoretic language of Jacquet-Langlands, Pizer's result  can be interpreted as giving a correspondence between automorphic forms on the adelization $B(\mathbb{A})$ and automorphic forms on $GL(2,\mathbb{A})$.
    
Recall that a left $O$-ideal $I$ is a lattice in $B$ such that $I_{p}=O_{p}a_{p}$ (for some $a_{p}\in B_{p}^{*}$) for every prime $p$. The set of all such ideals is denoted by $C$. Two left $O$-ideals $I$ and $J$ are equivalent if there exists $a\in B^{*}$ such that $I=Ja$. This gives rise to the set $\bar{C}=\{[I_{1}],[I_{2}],...,[I_{H}]\}$ of left $O$-ideal classes which has a finite cardinality $H$. We also define the norm of an ideal $I$ to be the positive rational number which generates the fractional ideal of $\mathbb{Q}$ generated by $\{N(x): x\in I\}$. Here $N(x)$ denotes the reduced norm of the quaternion element $x$.

Define the right order of a left $O$-ideal $I$ to be the set $O_{r}(I)= \{a\in B: Ia\subseteq I\}$; this is also an order of level $N$ in $B$.

Denote by $X$ the free abelian group with basis $\bar{C}$, the set of left $O$-ideal classes. We shall now illustrate the well-known construction of Hecke operators acting on the $\mathbb{R}$-vector space  $X(\mathbb{R})=X\otimes_{\mathbb{Z}}\mathbb{R}$. We define a height pairing $($ , $)$ on $X$ with integer values by setting $([I_{i}],[I_{j}])=0$ if $[I_{i}]\neq[I_{j}]$ and $([I_{i}],[I_{i}])=\frac{1}{2}\#O_{r}^{*}(I_{i})$, then extending bi-additively. This height pairing induces an inner product on $X(\mathbb{R})$. For each $n\geq1$, the Hecke operator $t_{n}:X(\mathbb{R})\rightarrow X(\mathbb{R})$ is defined by : $$t_{n}([I_{i}])=\sum_{j=1}^{H}(B(n))_{ij}[I_{j}].$$ The entries of the Brandt matrix $B(n)$ are calculated using $$(B(n))_{ij}=\frac{1}{e_{j}}\times\# \{x\in {I_{j}}^{-1}I_{i}:\frac{N(x)}{N({I_{j}}^{-1}I_{i})}=n\},$$ where $$e_{j}=\#\{x\in O_{r}(I_{j}): N(x)=1\}=\#O_{r}^{*}(I_{j}).$$
In other words, $e_{j}(B(n))_{ij}$ is the $n^{th}$ coefficient in the Fourier expansion of the theta series $$\theta_{ij}{(z)=\sum_{x\in {I_{j}}^{-1}I_{i}}q^{\frac{N(x)}{N({I_{j}}^{-1}I_{i})}}}\text{ }\text{ }\text{ }\text{ }\text{ }\text{ }(q=exp(2\pi i z)).$$ In \cite{Pizer}, Pizer described an explicit algorithm to compute these matrices. The main procedure computes the number of times $Q(x)=\frac{N(x)}{N(I)}$ represents $1,2,3,..,T$ for some given $T$ as $x$ varies  over the lattice $I$ in our quaternion algebra. The graph of $Q(x)$ with $x\in\mathbb{R}^{4}$ is a $4$-dimensional paraboloid which has a unique minimum point, and as we move away from this point in any direction, the values given by $Q(x)$ will always increase. That was the very simple  idea behind Pizer's method. 

The Hecke operators $t_{n}$ generate a commutative ring $\mathbb{T}$ of self adjoint operators $($\cite{Pizer}, Proposition $2.22$$)$. The spectral theorem implies that $X(\mathbb{R})$ has an orthogonal basis of eigenvectors for $\mathbb{T}$. 

As outlined in the introduction, we shall now construct a non-zero modular form $g$ in the Kohnen subspace of  $S_{\frac{3}{2}}(\Gamma_{0}(4N),F)$ as a linear combination of theta series generated by the norm form of $B$ evaluated on some ternary lattices. Needless to say, the ideas here are all well-known and largely drawn from articles \cite{gross} and \cite{BoSc}. Our main result is in Section 3 where we use the form $g$ to compute yet another half-integral weight modular form $h$ which also maps to $F$ via the Shimura correspondence. These two forms make up a basis for the space $S_{\frac{k}{2}}(\Gamma_{0}(N),F)$. 

For every left $O$-ideal $I_{i}$ in $C$, let $O_{i}$ be its right order, and let $R_{i}$ be the subgroup of trace zero elements in the suborder $\mathbb{Z}+2O_{i}.$ For every left $O$-ideal class $[I_{i}]$, we associate the ternary theta series \begin{equation} \label{eq:g}
g([I_{i}])=\frac{1}{2}\sum_{x\in R_{i}}q^{N(x)}=\frac{1}{2}\sum_{D\geq0}a_{D}([I_{i}])q^{D},
\end{equation} then extend this association by linearity to $X(\mathbb{R})$. Since, the ternary quadratic form $N(x)$ for $x \in R_{i}$ is a positive definite integral quadratic form with level $4N$ and a square discriminant, these modular forms have weight $\frac{3}{2}$, level $4N$ and the trivial character. 

The determination of the theta series $g([I_{i}])$ up to a precision $T$ amounts to computing the number of times  $N(x)$ represents $1,2,3,...,T$ as $x$ varies over all trace zero elements in $\mathbb{Z}+2O_{i}$. Therefore, it takes time roughly proportional to $T^{\frac{3}{2}}$. We compute these theta series by implementing a method similar to that of Pizer in \cite{Pizer} $($see Section 4 for examples$)$.

The action of the (half-integral) Hecke operators $T_{p^{2}}$ on the theta series $g(I)$ for $I\in X(\mathbb{R})$ is compatible with the action of the Hecke operators $t_{n}$ on $I$. More precisely, one can show that $T_{p^{2}}(g(I))=g(t_{p}(I))$ for all $p\nmid4N$ and all $I\in X(\mathbb{R})$ $($see Proposition 1.7 in \cite{tornaria-pacetti}$)$.
\begin{theorem}\label{quadratictwist}
$($\cite{BoSc} Theorem 3.2$)$ If $I_{F}$ is a non-zero element in the $F$-isotypical component of $X(\mathbb{R})$, then $g=g(I_{F})$ is in the Kohnen subspace of $S_{\frac{3}{2}}(\Gamma_{0}(4N),F)$, and $g(I_{F})$ is non-zero if and only if $L(F,\frac{1}{2})\neq0$. \end{theorem} 
We compute the eigenvector $I_{F}$ by using the Brandt module package in Sage. Hence, if $\displaystyle{I_{F}=\sum_{i=1}^{H}e_{i}[I_{i}]}$, then $g$ is computed as the linear combination $\displaystyle{\sum_{i=1}^{H}e_{i}g([I_{i}])}$. 

The above procedure can be generalized to compute modular forms of higher weights by using what are known as generalized theta series. Given a newform $F$ in $S_{k-1}(\Gamma_{0}(N))$ with an odd square-free level $N$ such that $k\equiv3\mod{4}$ and $L(F,\frac{1}{2})\neq 0$, the work of B\"{o}cherer and Schulze-Pillot guarantees the existence of a non-zero modular form $g=g_{F}$ in the Kohnen subspace of $S_{\frac{k}{2}}(\Gamma_{0}(4N),F)$. The form $g$ is obtained as a linear combination of generalized theta series attached to the ternary quadratic form $N(x)$ and some homogeneous harmonic polynomials of degree $\frac{k-3}{2}$ in three variables. The reader is referred to \cite{BoSc} for a thorough discussion of this construction. 

\section{The Shimura Correspondence and Theta Lifts}

In the preceding treatment, we have only considered half-integral weight modular forms with the trivial character. A naturally arising problem is to compute a half-integral weight modular form $g$ with a quadratic character $\boldsymbol{\chi}$, which belongs to the Kohnen subspace and maps to $F$ via the Shimura correspondence. 
Unfortunately, the above classical construction does not generalize in any obvious or simple way to yield such a form. In theory, however, we are  guaranteed the existence of $g$ $($under some restrictions on $\boldsymbol{\chi}$$)$ by a beautiful result of Baruch and Mao $($Theorem 10.1 in \cite{BaMa}$)$ which, in parts, states the following. 

Consider a newform $F$ of even weight $k-1$, odd square-free level $N$ and of trivial character. Let $S_{N}$ be the set of primes dividing $N$, and let $S$ be a subset of $S_{N}$. Write $N'=\prod_{p\in S}p$, and let $\boldsymbol{\chi}=\prod_{p|2N}\boldsymbol{\chi}_{(p)}$ be any even Dirichlet character defined modulo $4NN'$ such that $\boldsymbol{\chi}_{(p)}\equiv1$ when $p | (N/N')$ and $\boldsymbol{\chi}_{(p)}(-1)=-1$ when $p|N'$. Set $\chi=\boldsymbol{\chi}.\chi_{-1}^{s}$ where $\chi_{-1}$ is the Dirichlet character modulo $4$ defined by $\chi_{-1}(n) =\left(\frac{-1}{n}\right)$ and $s$ is the size of $S$. In particular, $\chi$ is unramified at the prime $2$. There exists a unique (up to scalar multiple) cusp form $g_{S}$ that is in the Kohnen subspace of $S_{\frac{k}{2}}(4NN',\boldsymbol{\chi},F_{\chi})$. The squares of the Fourier coefficients of $g_{S}$ are related to the central values of the quadratic twists of the $L$-function of $F$. 

In \cite{BaMa}, the construction of $g_{S}$ is of adelic nature because the authors worked in the setting of automorphic representations. A careful translation from the adelic language to the modular form language is needed to write down an explicit expression of $g_{S}$ as a linear combination of classical theta series. We hope to pursue this work in a future paper.

Before we proceed into the next section, some remarks on the representation theoretic version of the Shimura correspondence seem to be in order. The theta correspondence for the pairs $(\widetilde{SL_{2}}, PB^{*})$ and $(\widetilde{SL_{2}},PGL_{2})$ lies at the heart of the above classical constructions. Following Waldspurger's treatment in \cite{waldspurger1} and \cite{waldspurger}, the half-integral weight modular form $g$ can be realized as a theta lift from $PGL_{2}(\mathbb{A})$ to $\widetilde{SL_{2}(\mathbb{A})}$, the two-fold metaplectic cover of $SL_{2}(\mathbb{A})$. Moreover, one can obtain $g$ as a theta lift from $PB^{*}(\mathbb{A})$ to $\widetilde{SL_{2}(\mathbb{A})}$ $($see \cite{waldspurger2}$)$. To clarify things in the reader's mind, let us discuss these ideas briefly.
 
Let $\tilde{A}'_{\frac{k}{2}}(M,\chi_{0})$ denote the space of cuspidal automorphic forms $\tilde{\phi}=\otimes_{v}\tilde{\phi}_{v}$ on $S_{\mathbb{Q}}\backslash\widetilde{SL_{2}(\mathbb{A})}$ with character $\chi_{0}=\boldsymbol{\chi}.\chi_{-1}^{\frac{k-1}{2}}$ and level $M$. This space is determined by a set of local conditions satisfied by the vectors $\tilde{\phi}_{v}$ $($see \cite{waldspurger} pp. 381--388 for notation and other details$)$. It is known that there is a natural bijection between $S'_{\frac{k}{2}}(M,\boldsymbol{\chi})$ and $\tilde{A}'_{\frac{k}{2}}(M,\chi_{0})$. Thus, half-integral weight modular forms may be viewed as automorphic forms on $S_{\mathbb{Q}}\backslash\widetilde{SL_{2}(\mathbb{A})}$ whenever convenient. In this framework, the Shimura correspondence has the following formulation. Suppose that $g\mapsto F$ under the Shimura map. Let $\tilde{\pi}$ be the automorphic representation of $S_{\mathbb{Q}}\backslash\widetilde{SL_{2}(\mathbb{A})}$ associated to $g$, and let $\pi$ be the automorphic representation of $PGL_{2}(\mathbb{A})$ associated to $F$. Then, $S_{\psi}(\tilde{\pi})=\pi\otimes\chi_{0}^{-1}$, where $\psi$ is the usual additive character on $\mathbb{A}/\mathbb{Q}$. The space $S_{\psi}(\tilde{\pi})$ is defined in \cite{waldspurger1} as $S_{\psi}(\tilde{\pi})=\Theta(\tilde{\pi},\psi^{\nu})\otimes\chi_{\nu}$ for any choice of $\nu\in\mathbb{Q}^{*}$ such that $\Theta(\tilde{\pi},\psi^{\nu})\neq0$.

Denote  by $\pi'$ the automorphic representation of $PB^{*}$ associated to $\pi$ via the Jacquet-Langlands correspondence.  In the previous section, the form $g$ is obtained as a special vector in the representation space of $\Theta(\pi',\psi)$ using the theta correspondence for the pair $(\widetilde{SL_{2}}, PB^{*})$. Notice that $g$ coincides (up to scalar multiple) with the form $g_{S}$ for $S=\phi$. This follows directly from the theorem described in the second paragraph of the present section since $N'=1$ if $\boldsymbol{\chi}$ is the trivial character.

We shall now shed some light on the construction of the form $g_{S}$ for any $S$ as carried out in \cite{BaMa}. Using the theta correspondence of the pair $(\widetilde{SL_{2}},PGL_{2})$, the form $g_{S}$ is obtained as a special vector in the representation space $V_{\tilde{\pi}^{D}}$ of $$\Theta(\pi\otimes\chi_{D},\psi^{|D|})=\tilde{\pi}^{D}=\otimes_{v}\tilde{\pi}^{D}_{v}$$ for some fundamental discriminant $D$ determined by $S$.

Let $h$ be a non-zero modular form in  $S_{\frac{k}{2}}(4NN',\boldsymbol{\chi},F_{\chi})$, and denote by $\tilde{\pi}_{h}$ the automorphic representation of $S_{\mathbb{Q}}\backslash\widetilde{SL_{2}(\mathbb{A})}$ generated by $h$. One gets that $\tilde{\pi}_{h}=\tilde{\pi}_{g_{S}}=\tilde{\pi}^{D}$. Moreover, the space $V'_{\tilde{\pi}^{D}}=\tilde{A}_{\frac{k}{2}}(4NN',\chi_{0})\cap V_{\tilde{\pi}^{D}}$ is two-dimensional. In the next section, we verify this fact for $S=\phi$. Roughly speaking, this owes to the fact that the space of vectors in $\tilde{\pi}^{D}_{2}$ that satisfy the $2$-adic condition forced on a form $\tilde{\phi}\in\tilde{A}_{\frac{k}{2}}(4NN',\chi_{0})$ is two-dimensional. This is the space $V'_{{\tilde{\pi}^{D},2}}$ in the factorization $V'_{\tilde{\pi}^{D}}=\otimes_{v}V'_{\tilde{\pi}^{D},v}$. Moreover, if we let $\tilde{A}^{+}_{\frac{k}{2}}(4NN',\chi_{0})$ be the image of the Kohnen subspace under the natural bijection $S_{\frac{k}{2}}(4NN',\boldsymbol{\chi})\rightarrow\tilde{A}_{\frac{k}{2}}(4NN',\chi_{0})$, we get that $\tilde{A}^{+}_{\frac{k}{2}}(4NN',\chi_{0})\cap V_{\tilde{\pi}^{D}}$ is one-dimensional, and $g_{S}$ is a non-zero vector in this space. In fact, $\tilde{A}^{+}_{\frac{k}{2}}(4NN',\chi_{0})$ turns out to be the space of forms $\tilde{\phi}=\otimes_{v}\tilde{\phi}_{v}$ in $\tilde{A}_{\frac{k}{2}}(4NN',\chi_{0})$ with $\tilde{\phi}_{2}$ being a carefully chosen vector in $V'_{{\tilde{\pi}^{D},2}}$ $($see Section 9.4 in \cite{BaMa}$)$. It is exactly this choice of $\tilde{\phi}_{2}$ that forces $g_{S}$ to lie in the Kohnen subspace.

\section{On The Result Of Waldspurger}

First, we state Waldspurger's theorem in its most general form. Let $F\in S_{k-1}^{\mathrm{new}}(\boldsymbol{\chi}^{2})$, and let $\pi$ be the irreducible automorphic representation of $GL_{2}(\mathbb{A})$ associated to $F$. The newform $F$ is required to satisfy the following condition: when $\pi_{p}=\pi(\mu_{1,p},\mu_{2,p})$ belongs to the principal series, then both characters $\mu_{1,p}$ and $\mu_{2,p}$ are even. Notice that $\mu_{1,p}(-1)=\mu_{2,p}(-1)=1$ whenever $\pi_{p}$ is unramified principal series. In fact, Flicker proved that this condition is satisfied if and only if $S_{\frac{k}{2}}(L,\boldsymbol{\chi},F)\neq 0$ for some positive integer $L$. We also require that one of the following conditions is satisfied:
\begin{enumerate}
\item{The level of $F$ is divisible by $16$}
\item{The conductor of the character $\chi_{0}$ is divisible by $16$ $($$\chi_{0}(n)=\boldsymbol{\chi}(n)\left(\frac{-1}{n}\right)^{\frac{k-1}{2}}$$)$}
\item{$\pi_{2}$ is not supercuspidal}
\end{enumerate}
For every positive integer $e$ and prime number $p$, Waldspurger used the local components of $\pi$ and $\chi_{0}$ to define an integer $n_{p}$ and a set $U_{p}(e,F)$ of finitely many functions $c_{p}:\mathbb{Q}_{p}^{*}\rightarrow\mathbb{C}$. According to these definitions, given an integer $E\geq1$, for all but finitely many $p$ the set $U_{p}(v_{p}(E),F)$ is a singleton and $n_{p}=0$. In fact, we are interested in the finite set $\displaystyle{U(E,F)=\prod_{p\text{ }prime}U_{p}(v_{p}(E),F)}$. 

Let $\mathbb{N}^{\mathrm{sf}}$ be the set of positive square-free integers, and let $A:\mathbb{N}^{\mathrm{sf}}\rightarrow\mathbb{C}$ be any function. Denote square-free part of a positive integer $n$ by $n^{\mathrm{sf}}$. Given an element $c_{E}=(c_{E,p})\in U(E,F)$, we define the function $f(c_{E},A)$ on $\mathbb{H}$ by: $$ f(c_{E},A)=\sum_{n\geq1}a_{n}(c_{E},A)q^{n},$$ where $$a_{n}(c_{E},A)=A(n^{\mathrm{sf}})n^{\frac{k-2}{4}}\prod_{p}c_{E,p}(n).$$
The integer $\displaystyle{\prod_{p}p^{n_{p}}}$ is denoted by $N(F)$, and the complex vector space generated by the set $\{f(c_{E},A_{F})\}_{c_{E}\in U(E,F)}$ is denoted by $U(E,F,A_{F})$.

Given a Dirichlet character $\nu$, let $L(\nu,s)$ be the associated $L$-function. Recall that, for sufficiently large $\operatorname{Re}s$, $L(\nu,s)$ is defined as $$L(\nu,s)=\pi^{-\frac{s+\delta}{2}}\Gamma\left(\frac{s+\delta}{2}\right)\sum_{n=1}^{\infty}\hat{\nu}(n)n^{-s},$$ where $\delta$ is such that $\nu(-1)=(-1)^{\delta}$ and $\hat{\nu}$ is the primitive Dirichlet character associated to $\nu$. It is known that $L(\nu,s)$ has a meromorphic continuation to all $s$. Define $\epsilon(\nu,s)$ as the $\epsilon$-factor that appears in the functional equation $L(\nu^{-1},1-s)=\epsilon(\nu,s)L(\nu,s)$.

 For any $t\in\mathbb{Z}^{*}$, denote by $\chi_{t}$ the quadratic Dirichlet character associated with the extension $\mathbb{Q}(\sqrt{t})/\mathbb{Q}$. Notice that $\chi_{t}(n)=(\frac{\Delta_{t}}{n})$, where $\Delta_{t}$ is the discriminant of $\mathbb{Q}(\sqrt{t})$. In particular, $\chi_{t}(n)=1$ for all $n\in \mathbb{Z}^{*}$ if $t$ is a perfect square.

\begin{theorem}\label{waldspurger2}
$($See \cite{waldspurger}$)$ Given a newform $F$ satisfying the above hypotheses, there exists a function $A_{F}:\mathbb{N}^{\mathrm{sf}}\rightarrow\mathbb{C}$, such that:
\begin{enumerate}
\item{$(A_{F}(t))^{2}=L(F\otimes\chi_{0}^{-1}\chi_{t},\frac{1}{2})\epsilon(\chi_{0}^{-1}\chi_{t},\frac{1}{2})$} 
\item{$S_{\frac{k}{2}}(4N,F,\chi)=\bigoplus{U(E,F,A_{F})}$; the direct sum being taken over all integers $E$ such that $N(F)|E|4N$}
\end{enumerate}
\end{theorem}
Our goal is thus reduced to computing the individual spaces $U(E,F,A_{F})$. This amounts to computing the half-integral weight modular forms $f(c_{E},A_{F})$ for all $c_{E}\in U(E,F)$. Recall that given $c_{E}=(c_{E,p})\in U(E,F)$, the function $f(c_{E},A_{F})$ is written as $\displaystyle{\sum_{n\geq1}a_{n}(c_{E},A_{F})q^{n}}$, and the Fourier coefficients $a_{n}(c_{E},A_{F})$ are defined by the formula $$a_{n}(c_{E},A_{F})=A_{F}(n^{\mathrm{sf}})n^{\frac{k-2}{4}}\prod_{p}c_{E,p}(n).$$  The local factor $\displaystyle{\prod_{p}c_{E,p}(n)}$ is completely determined by the local properties of $F$ and $\chi_{0}$. In fact, the various local objects involved in this calculation are given explicitly in \cite{waldspurger}. However, the determination of the global factor $A(n^{\mathrm{sf}})$ requires more effort because we are only given its square value. One way to resolve this complication is to use the Fourier coefficients of the half-integral weight modular form $g$ that was computed in Section $1$. 

Let us now return to our usual setting. Recall that we are given a newform $\displaystyle{F=\sum_{n\geq1}b_{n}q^{n}}$ in $S_{k-1}(\Gamma_{0}(N))$ with odd square-free level $N$ such that $k\equiv3\mod{4}$ and $L(F,\frac{1}{2})\neq0$. Let $\displaystyle{g=\sum_{n\geq1}a_{n}q^{n}\in S_{\frac{k}{2}}(\Gamma_{0}(4N),F)}$ be the form obtained in Section 1.
 
For every prime number $p$ not dividing $N$, we put $\lambda_{p}=b_{p}p^{1-\frac{k}{2}}$, $\alpha_{p}+\alpha_{p}'=\lambda_{p}$, and $\alpha_{p}\alpha_{p}'=1$. We also set $\lambda_{p}'=b_{p}p^{1-\frac{k}{2}}$ for every prime divisor $p$ of $N$. We shall now apply the theorem of Waldspurger to compute a basis for the space $S_{\frac{k}{2}}(\Gamma_{0}(4N),F)$. In order to compute the local factors $\displaystyle{\prod_{p}c_{p}}$ of the Fourier coefficients of the desired basis elements, we use the explicit formulae given in \cite{waldspurger} Section $8$.

Since the newform $F$ has a square-free level and the trivial character modulo $N$, we get 
$$n_{p}=\begin{cases}
2 & \text{if }p=2;\\
1 & \text{if }p|N;\\
0 & \text{otherwise.}
\end{cases}$$
Hence, $N(F)=\prod p^{n_{p}}=4N$. We also get $$S_{\frac{k}{2}}(\Gamma_{0}(4N),F)=\bigoplus_{N(F)|E|4N}{U(E,F,A_{F})}=U(4N,F,A_{F}),$$ where $U(E,F,A_{F})$ is the vector space generated by the modular forms $\{f(c_{E},A_{F})\}_{c_{E}\in U(E,F)}$, and $U(E,F)$ is the finite set $\displaystyle{\prod_{p}U_{p}(v_{p}(E),F)}$. Recall that for a given prime number $p$ and a positive integer $E$, the set $U_{p}(v_{p}(E),F)$ consists of finitely many local functions $c_{p}$. Using the fact that $E=4N$ and $N$ is odd and square-free, we get
$$U_{2}(2,F)=\begin{cases}
\{c_{2}'[\alpha_{2}],c_{2}'[\alpha_{2}']\} & \text{if }\alpha_{2}\neq\alpha_{2}';\\
\{c_{2}'[\alpha_{2}],c_{2}''[\alpha_{2}]\}& \text{otherwise,}\\
\end{cases}$$
where the local functions $c_{2}'$ and $c_{2}''$ are defined below.

If $p$ is an odd prime we get,
$$U_{p}(v_{p}(4N),F)=\begin{cases}
\{c_{p}^{0}[\lambda_{p}]\} & \text{if }p\nmid N;\\
\{c_{p}^{s}[\lambda_{p}']\}& \text{otherwise. }\\
\end{cases}$$
Before we proceed any further, we need to recall the definition of the Hilbert symbol. Given two non-zero elements $a$ and $b$ in a local field $K$, the Hilbert symbol is defined by $$(a,b)=\begin{cases}
1 & \text{if }z^{2}=ax^{2}+by^{2}\text{ has a non-zero solution in } K^{3};\\
-1& \text{otherwise.}\\
\end{cases}$$
In particular, if $a, b\in \mathbb{Z}$, write $a=2^{\alpha}s$ and $b=2^{\beta}t$ such that $s$ and $t$ are odd integers. The Hilbert symbol over the 2-adics is $(a,b)_{2}=(-1)^{\epsilon(s)\epsilon(t)+\alpha\omega(t)+\beta\omega(s)},$ where $\epsilon(x)=\frac{x-1}{2}$ and $\omega(x)=\frac{x^{2}-1}{8}$.

In order to evaluate the local functions $c_{p}$ at an integer $n$, we need to express $n$ as $up^{2h}$ ($v_{p}(u)\in\{0,1\}$), then use the following formulae:
$$c'_{2}[\delta](n)=\begin{cases}
\delta^{h}(\delta-2^{-\frac{1}{2}}(2,u)_{2}) & \text{if }v_{2}(u)=0\text{ and }(u,-1)_{2}=-1;\\
\delta^{h}, & \text{otherwise.}
\end{cases}$$ 
$$c_{2}''[\delta](n)=\begin{cases}
\delta & \text{if }v_{2}(n)=0\text{ and }(n,-1)_{2}=-1;\\
0 & \text{if }v_{2}(n)=1 \text{ or }(v_{2}(n)=0\text{ and }(n,-1)_{2}=1);\\
\delta(c_{2}''[\delta](\frac{n}{4})+c_{2}'[\delta](\frac{n}{4})) & \text{otherwise.}
\end{cases}$$
$$c_{p}^{s}[\delta](n)=\begin{cases}
\delta^{h} & \text{if }v_{p}(u)=1;\\
2^{\frac{1}{2}}\delta^{h} & \text{if }v_{p}(u)=0\text{ and }\left(\frac{-u}{p}\right)=-p^{\frac{1}{2}}\delta;\\
0, & \text{otherwise.}
\end{cases}$$ 
\begin{equation}\label{cp0}c_{p}^{0}[\delta](n)=\begin{cases}
1 & \text{if }v_{p}(u)=0\text{ and }h=0;\\
b_{h}-b_{h-1}\left(\frac{-u}{p}\right)p^{-\frac{1}{2}} & \text{if }v_{p}(u)=0\text{ and }h\geq1;\\
b_{h} & \text{otherwise,}
\end{cases}\end{equation} where
$$b_{h}=\frac{1}{2^{h}}\left(\sum_{i=0}^{\lfloor\frac{h}{2}\rfloor}\binom{h+1}{2i+1}\delta^{h-2i}(\delta^{2}-4)^{i}\right).$$
If $\alpha_{2}\neq\alpha_{2}'$, the set $U(4N,F)$ consists of $$\displaystyle{\mathbf{c_{1}}=c_{2}'[\alpha_{2}]\prod_{p|N}c_{p}^{s}[\lambda_{p}']\prod_{p\nmid2N}c_{p}^{0}[\lambda_{p}]} \text{ }\text{ }\text{and}\text{ }\text{ } \displaystyle{\mathbf{c_{2}}=c_{2}'[\alpha_{2}']\prod_{p|N}c_{p}^{s}[\lambda_{p}']\prod_{p\nmid2N}c_{p}^{0}[\lambda_{p}]}.$$
If $\alpha_{2}=\alpha_{2}'$, the set $U(4N,F)$ consists of $$\displaystyle{\mathbf{c_{1}}=c_{2}'[\alpha_{2}]\prod_{p|N}c_{p}^{s}[\lambda_{p}']\prod_{p\nmid2N}c_{p}^{0}[\lambda_{p}]} \text{ }\text{ }\text{and}\text{ }\text{ } \displaystyle{\mathbf{c'_{2}}=c_{2}''[\alpha_{2}]\prod_{p|N}c_{p}^{s}[\lambda_{p}']\prod_{p\nmid2N}c_{p}^{0}[\lambda_{p}]}.$$ Theorem ~\ref{waldspurger2} gives a basis for $S_{\frac{k}{2}}(\Gamma_{0}(4N),F)$ that consists of the functions $$f(\mathbf{c_{1}},A_{F})=\sum_{n\geq1}a_{n}(\mathbf{c_{1}},A_{F})q^{n}\text{ and }f(\mathbf{c_{2}},A_{F})=\sum_{n\geq1}a_{n}(\mathbf{c_{2}},A_{F})q^{n}\text{ if }\alpha_{2}\neq\alpha_{2}'$$ and $$f(\mathbf{c_{1}},A_{F})=\sum_{n\geq1}a_{n}(\mathbf{c_{1}},A_{F})q^{n}\text{ and }f(\mathbf{c'_{2}},A_{F})=\sum_{n\geq1}a_{n}(\mathbf{c'_{2}},A_{F})q^{n}\text{ if }\alpha_{2}=\alpha_{2}'.$$ Recall that $$a_{n}(\mathbf{c_{1}},A_{F})=A_{F}(n^{\mathrm{sf}})n^{\frac{k-2}{4}}c_{2}'[\alpha_{2}](n)\prod_{p|N}c_{p}^{s}[\lambda_{p}'](n)\prod_{p\nmid2N}c_{p}^{0}[\lambda_{p}](n)$$ $$a_{n}(\mathbf{c_{2}},A_{F})=A_{F}(n^{\mathrm{sf}})n^{\frac{k-2}{4}}c_{2}'[\alpha_{2}'](n)\prod_{p|N}c_{p}^{s}[\lambda_{p}'](n)\prod_{p\nmid2N}c_{p}^{0}[\lambda_{p}](n)$$ and $$a_{n}(\mathbf{c'_{2}},A_{F})=A_{F}(n^{\mathrm{sf}})n^{\frac{k-2}{4}}c_{2}''[\alpha_{2}](n)\prod_{p|N}c_{p}^{s}[\lambda_{p}'](n)\prod_{p\nmid2N}c_{p}^{0}[\lambda_{p}](n).$$

\begin{lemma}
 For every positive integer $n$ such that $(-1)^{\frac{k-1}{2}}n=-n\equiv2,3\mod{4}$, we have $a_{n}(\mathbf{c_{1}},A_{F})=a_{n}(\mathbf{c_{2}},A_{F})$ and $a_{n}(\mathbf{c'_{2}},A_{F})=0$.
\end{lemma}
\begin{proof}
Write $n$ as $u2^{2h}$ with $v_{2}(u)\in\{0,1\}$.
If $n\equiv1\mod{4}$, we get $c_{2}'[\alpha_{2}](n)=(\alpha_{2})^{0}=1=(\alpha_{2}')^{0}=c_{2}'[\alpha_{2}'](n)$ and $c_{2}''[\alpha_{2}](n)=0$ since $h=0$, $v_{2}(u)=0$ and $(u,-1)_{2}=1$. If $n\equiv2\mod{4}$, we also get $c_{2}'[\alpha_{2}](n)=(\alpha_{2})^{0}=1=(\alpha_{2}')^{0}=c_{2}'[\alpha_{2}'](n)$ and $c_{2}''[\alpha_{2}](n)=0$ since $h=0$ and $v_{2}(u)=1$. Thus, $a_{n}(\mathbf{c_{1}},A_{F})=a_{n}(\mathbf{c_{2}},A_{F})$ and $a_{n}(\mathbf{c'_{2}},A_{F})=0$ for all $n\equiv1,2\mod{4}$.
\end{proof}
\begin{lemma}
 If $\alpha_{2}\neq\alpha_{2}'$, the function $f(\mathbf{c_{1}},A_{F})-f(\mathbf{c_{2}},A_{F})$ belongs to the Kohnen subspace of $S_{\frac{k}{2}}(\Gamma_{0}(4N),F)$. Otherwise, the function $f(\mathbf{c'_{2}},A_{F})$ belongs to the Kohnen subspace of $S_{\frac{k}{2}}(\Gamma_{0}(4N),F)$
 \end{lemma}
 \begin{proof}
The coefficients $a_{n}(\mathbf{c_{1}},A_{F})-a_{n}(\mathbf{c_{2}},A_{F})$ and $a_{n}(\mathbf{c'_{2}},A_{F})$ are zero for all $n$ such that $(-1)^{\frac{k-1}{2}}n\equiv2,3\mod{4}.$
\end{proof}
\begin{proposition}
If $\alpha_{2}\neq\alpha_{2}'$, the half-integral weight modular form $g$ in Theorem $1.1$ is a non-zero scalar multiple of $f(\mathbf{c_{1}},A_{F})-f(\mathbf{c_{2}},A_{F})$. Otherwise, $g$ is a non-zero scalar multiple of $f(\mathbf{c'_{2}},A_{F})$.
\end{proposition}
\begin{proof}
We know by the multiplicity one result of Kohnen $($Theorem 2 in \cite{kohnen}$)$ that the Kohnen subspace of $S_{\frac{k}{2}}(\Gamma_{0}(4N),F)$ is one-dimensional. We also know from Section 1 that $g$ belongs to the Kohnen subspace. This forces $g$ to be equal to a non-zero scalar multiple of $f(\mathbf{c_{1}},A_{F})-f(\mathbf{c_{2}},A_{F})$ when $\alpha_{2}\neq\alpha_{2}'$ and a non-zero scalar multiple of $f(\mathbf{c'_{2}},A_{F})$ when $\alpha_{2}=\alpha_{2}'$.
\end{proof}

Let $t$ be a positive square-free integer and denote by $\Delta_{-t}$ the fundamental discriminant corresponding to $-t$. Recall that $g=\displaystyle{\sum_{n\geq1}a_{n}q^{n}}$. We know by the previous proposition that $$a_{|\Delta_{-t}|}(\mathbf{c_{1}},A_{F})-a_{|\Delta_{-t}|}(\mathbf{c_{2}},A_{F})=r_{1}a_{|\Delta_{-t}|} \text{ }\text{ }\text{ }\text{  if }\alpha_{2}\neq\alpha_{2}'$$ and 
$$a_{|\Delta_{-t}|}(\mathbf{c'_{2}},A_{F})=r_{2}a_{|\Delta_{-t}|} \text{ }\text{ }\text{ } \text{   if }\alpha_{2}=\alpha_{2}'$$ for some non-zero complex constants $r_{1}$ and $r_{2}$ such that $f(\mathbf{c_{1}},A_{F})-f(\mathbf{c_{2}},A_{F})=r_{1}g$ if $\alpha_{2}\neq\alpha_{2}'$ and $f(\mathbf{c'_{2}},A_{F})=r_{2}g$ if $\alpha_{2}=\alpha_{2}'$. On the other hand, we also know that $$a_{|\Delta_{-t}|}(\mathbf{c_{1}},A_{F})-a_{|\Delta_{-t}|}(\mathbf{c_{2}},A_{F})=A_{F}(t)|\Delta_{-t}|^{\frac{k-2}{4}}\left(c_{2}'[\alpha_{2}](|\Delta_{-t}|)-c_{2}'[\alpha_{2}'](|\Delta_{-t}|)\right)\prod_{p|N}c_{p}^{s}[\lambda_{p}'](|\Delta_{-t}|)$$ and $$a_{|\Delta_{-t}|}(\mathbf{c'_{2}},A_{F})=A_{F}(t)|\Delta_{-t}|^{\frac{k-2}{4}}c_{2}''[\alpha_{2}](|\Delta_{-t}|)\prod_{p|N}c_{p}^{s}[\lambda_{p}'](|\Delta_{-t}|).$$ Notice that the local factors for $\displaystyle{p\nmid2N}$ do not appear in the above expressions. In fact, $c_{p}^{0}[\lambda_{p}](|\Delta_{-t}|)=1$ for all such $p$ since $|\Delta_{-t}|$ is not divisible by the square of any odd prime. 

Hence,
$$A_{F}(t)|\Delta_{-t}|^{\frac{k-2}{4}}\left(c_{2}'[\alpha_{2}](|\Delta_{-t}|)-c_{2}'[\alpha_{2}'](|\Delta_{-t}|)\right)\prod_{p|N}c_{p}^{s}[\lambda_{p}'](|\Delta_{-t}|)=r_{1}a_{|\Delta_{-t}|}\text{ }\text{ }\text{ }\text{if }\alpha_{2}\neq\alpha_{2}',$$ and $$A_{F}(t)|\Delta_{-t}|^{\frac{k-2}{4}}c''_{2}[\alpha_{2}](|\Delta_{-t}|)\prod_{p|N}c_{p}^{s}[\lambda_{p}'](|\Delta_{-t}|)=r_{2}a_{|\Delta_{-t}|}\text{ }\text{ }\text{ }\text{if }\alpha_{2}=\alpha_{2}'.$$ Since $\Delta_{-t}$ is a fundamental discriminant, it is not divisible by any square of any odd prime, and it satisfies $\Delta_{-t}\equiv1\mod{4}$ or $\Delta_{-t}\equiv8,12\mod{16}$. Thus, $$c_{2}'[\alpha_{2}](|\Delta_{-t}|)-c_{2}'[\alpha_{2}'](|\Delta_{-t}|)=\alpha_{2}-\alpha_{2}'$$ and $$c''_{2}[\alpha_{2}](|\Delta_{-t}|)=\alpha_{2}.$$We still need to investigate the value of $\prod_{p|N}c_{p}^{s}[\lambda_{p}'](|\Delta_{-t}|)$. Since $|\Delta_{-t}|$ is not divisible by the square of any odd prime $c_{p}^{s}[\lambda_{p}'](|\Delta_{-t}|)=1$ for all $p$ such that $p|N$ and $p|\Delta_{-t}$. If $p|N$ and $p\nmid\Delta_{-t}$ then $$c_{p}^{s}[\lambda_{p}'](|\Delta_{-t}|)=\begin{cases}
2^{\frac{1}{2}} & \text{if }\left(\frac{\Delta_{-t}}{p}\right)=-p^{\frac{1}{2}}\lambda_{p}';\\
0 & \text{otherwise.}
\end{cases}$$ 
Therefore,  $$\prod_{p|N}c_{p}^{s}[\lambda_{p}'](|\Delta_{-t}|)=\begin{cases}
\displaystyle{\prod_{\substack{p|N\\p\nmid\Delta_{-t}}}2^{\frac{1}{2}}} & \text{if }\left(\frac{\Delta_{-t}}{p}\right)=-p^{\frac{1}{2}}\lambda_{p}'\text{ }\forall p,\text{ }p|N \text{ and }p\nmid\Delta_{-t};\\
0 & \text{otherwise.}
\end{cases}$$ 
Putting this together gives $$A_{F}(t)|\Delta_{-t}|^{\frac{k-2}{4}}(\alpha_{2}-\alpha_{2}')\prod_{\substack{p|N\\p\nmid\Delta_{-t}}}2^{\frac{1}{2}}=r_{1}a_{|\Delta_{-t}|}\text{ }\text{ }\text{ if }\alpha_{2}\neq\alpha_{2}'$$ and $$A_{F}(t)|\Delta_{-t}|^{\frac{k-2}{4}}\alpha_{2}\prod_{\substack{p|N\\p\nmid\Delta_{-t}}}2^{\frac{1}{2}}=r_{2}a_{|\Delta_{-t}|}\text{ }\text{ }\text{ if }\alpha_{2}=\alpha_{2}'$$whenever $\left(\frac{\Delta_{-t}}{p}\right)=-p^{\frac{1}{2}}\lambda_{p}'$ for all $p$ such that $p|N$ and $p\nmid\Delta_{-t}$. Hence we obtain the following proposition.
\begin{theorem}\label{prop}
For a positive square-free integer $t$ satisfying $\left(\frac{\Delta_{-t}}{p}\right)=-p^{\frac{1}{2}}\lambda_{p}'$ for all $p$ such that $p|N$ and $p\nmid\Delta_{-t}$, we have $$A_{F}(t)=r2^{\frac{\nu_{t}}{2}}|\Delta_{-t}|^{\frac{2-k}{4}}a_{|\Delta_{-t}|},$$ where $r$ is a non-zero complex constant depending only on $F$, and $\displaystyle{2^{\frac{\nu_{t}}{2}}=\prod_{\substack{p|N\\p\nmid\Delta_{-t}}}2^{-\frac{1}{2}}}.$
\end{theorem}
\begin{remark}
Let $n$ be a positive integer such that its square-free part which we denote by $n^{\mathrm{sf}}$ satisfies $\left(\frac{\Delta_{-n^{\mathrm{sf}}}}{p}\right)=p^{\frac{1}{2}}\lambda_{p}'$ for some prime $p$ such that $p|N$ and $p\nmid\Delta_{-n^{\mathrm{sf}}}$. For this particular $p$, write $n=up^{2h}$ ($v_{p}(u)\in\{0,1\}$). Since $p\nmid\Delta_{-n^{\mathrm{sf}}}$, we get $v_{p}(u)=0$. Moreover, $\left(\frac{-u}{p}\right)=\left(\frac{\Delta_{-n^{\mathrm{sf}}}}{p}\right)=p^{\frac{1}{2}}\lambda_{p}'$. This forces $c_{p}^{s}[\lambda_{p}']$ to vanish at $n$. Therefore, $a_{n}(\mathbf{c_{1}},A_{F})$, $a_{n}(\mathbf{c_{2}},A_{F})$, $a_{n}(\mathbf{c'_{2}},A_{F})$ vanish for all such $n$ regardless of the value of $A_{F}(n^{\mathrm{sf}})$. 
\end{remark}
Before stating our main theorem, we shall verify that $f(\mathbf{c_{1}},A_{F})$ and $f(\mathbf{c_{2}},A_{F})$ are not scalar multiples of each other even though they share the same Hecke eigenvalues at the primes $p\nmid N$. We assume that $\alpha_{2}\neq \alpha'_{2}$ purely for simplicity of exposition. Let $n$ be a positive integer such that $a_{n}(\mathbf{c_{1}},A_{F})\neq0$. In particular, $A_{F}(n^{\mathrm{sf}})\neq0$ and $c_{p}^{s}[\lambda'_{p}](n)\neq0$. We set $r=|\Delta_{-n^{\mathrm{sf}}}|$ and $w=4|\Delta_{-n^{\mathrm{sf}}}|$ to simplify notation. Hence, \begin{equation}\label{cps}\prod_{p\mid N}c_{p}^{s}[\lambda'_{p}](r)=\prod_{p\mid N}c_{p}^{s}[\lambda'_{p}](w)= \prod_{\substack{p|N\\p\nmid r}}2^{\frac{1}{2}}.\end{equation} Moreover, for $\delta\in\{\alpha_{2},\alpha'_{2}\}$ we get \begin{equation}\label{c2}c'_{2}[\delta](r)=\begin{cases}
\delta-2^{-\frac{1}{2}}(2,r)_{2} & \text{if }r\equiv3\mod4;\\
\delta & \text{otherwise,}\\
\end{cases}\end{equation} and
\begin{equation}\label{c2'}c'_{2}[\delta](w)=\begin{cases}
\delta(\delta-2^{-\frac{1}{2}}(2,r)_{2}) & \text{ if }r\equiv3\mod4;\\
\delta^{2} & \text{ otherwise.}\\
\end{cases}\end{equation}Notice that the coefficients of $q^{r}$ and $q^{w}$ in $f(\mathbf{c_{1}},A_{F})$ and $f(\mathbf{c_{2}},A_{F})$ are non-zero. The following table shows that these coefficients are not in a constant ratio to each other. In fact, $\frac{a_{w}(\mathbf{c_{1}},A_{F})}{a_{w}(\mathbf{c_{2}},A_{F})}= \frac{\alpha_{2}}{\alpha'_{2}}\left(\frac{a_{r}(\mathbf{c_{1}},A_{F})}{a_{r}(\mathbf{c_{2}},A_{F})}\right)$.
\begin{table}[ht]
\caption{Ratios of Fourier Coefficients}
\centering
\begin{tabular}{|c|c|c|}
\hline
$$ & $r\equiv3\mod4$ & $r\equiv4,8\mod16$ \\
\hline
$\frac{a_{r}(\mathbf{c_{1}},A_{F})}{a_{r}(\mathbf{c_{2}},A_{F})}$ &  $\frac{\alpha_{2}-2^{-\frac{1}{2}}(2,r)_{2}}{\alpha'_{2}-2^{-\frac{1}{2}}(2,r)_{2}}$ & $ \frac{\alpha_{2}}{\alpha'_{2}}$ \\
\hline
$\frac{a_{w}(\mathbf{c_{1}},A_{F})}{a_{w}(\mathbf{c_{2}},A_{F})}$ & $\frac{\alpha_{2}(\alpha_{2}-2^{-\frac{1}{2}}(2,r)_{2})}{\alpha'_{2}(\alpha'_{2}-2^{-\frac{1}{2}}(2,r)_{2})}$  & $(\frac{\alpha_{2}}{\alpha'_{2}})^{2}$\\
\hline
\end{tabular}
\end{table}\\

In what follows, the terms $K_{1}(n)$, $K_{2}(n)$ and $K'_{2}(n)$ refer to the formulae 
\begin{equation}\label{K1}K_{1}(n)=c'_{2}[\alpha_{2}](n)\prod_{p|N}c_{p}^{s}[\lambda_{p}'](n)\prod_{p\nmid2N}c_{p}^{0}[\lambda_{p}](n)\end{equation}  \begin{equation}\label{K2}K_{2}(n)=c'_{2}[\alpha_{2}'](n)\prod_{p|N}c_{p}^{s}[\lambda_{p}'](n)\prod_{p\nmid2N}c_{p}^{0}[\lambda_{p}](n)\end{equation} 
and \begin{equation}\label{K3}K'_{2}(n)=c''_{2}[\alpha_{2}](n)\prod_{p|N}c_{p}^{s}[\lambda_{p}'](n)\prod_{p\nmid2N}c_{p}^{0}[\lambda_{p}](n)\end{equation} 
\begin{theorem}\label{main}
Let $\displaystyle{F=\sum_{n\geq1}b_{n}q^{n}}$ be a newform in $S_{k-1}^{\mathrm{new}}(\Gamma_{0}(N))$ with odd square-free level $N$ such that $k\equiv3\mod{4}$ and $L(F,\frac{1}{2})\neq0$. Let $\displaystyle{g=\sum_{n\geq1}a_{n}q^{n}}\in S_{\frac{k}{2}}(\Gamma_{0}(4N),F)$ be the form obtained in Section 1. Put $\displaystyle{f_{1}=\sum_{n\geq1}a_{n}(f_{1})q^{n}}$, $\displaystyle{f_{2}=\sum_{n\geq1}a_{n}(f_{2})q^{n}}$ and $\displaystyle{f'_{2}=\sum_{n\geq1}a_{n}(f'_{2})q^{n}}$ with $$a_{n}(f_{1})=\begin{cases} 2^{\frac{\nu_{n}}{2}}a_{|\Delta_{-n^{\mathrm{sf}}}|}|\Delta_{-n^{\mathrm{sf}}}|^{\frac{2-k}{4}}n^{\frac{k-2}{4}}K_{1}(n) & \text{if } \left(\frac{\Delta_{-n^{\mathrm{sf}}}}{p}\right)=-p^{\frac{1}{2}}\lambda_{p}'\text{ }\forall p,\text{ }p|N \text {and }p\nmid\Delta_{-n^{\mathrm{sf}}};\\ 0 & \text{otherwise,}\end{cases}$$ $$a_{n}(f_{2})=\begin{cases} 2^{\frac{\nu_{n}}{2}}a_{|\Delta_{-n^{\mathrm{sf}}}|}|\Delta_{-n^{\mathrm{sf}}}|^{\frac{2-k}{4}}n^{\frac{k-2}{4}}K_{2}(n) & \text{if } \left(\frac{\Delta_{-n^{\mathrm{sf}}}}{p}\right)=-p^{\frac{1}{2}}\lambda_{p}'\text{ }\forall p,\text{ }p|N \text {and }p\nmid\Delta_{-n^{\mathrm{sf}}};\\ 0 & \text{otherwise,}\end{cases}$$ and $$a_{n}(f'_{2})=\begin{cases} 2^{\frac{\nu_{n}}{2}}a_{|\Delta_{-n^{\mathrm{sf}}}|}|\Delta_{-n^{\mathrm{sf}}}|^{\frac{2-k}{4}}n^{\frac{k-2}{4}}K'_{2}(n) & \text{if } \left(\frac{\Delta_{-n^{\mathrm{sf}}}}{p}\right)=-p^{\frac{1}{2}}\lambda_{p}'\text{ }\forall p,\text{ }p|N \text {and }p\nmid\Delta_{-n^{\mathrm{sf}}};\\ 0 & \text{otherwise.}\end{cases}$$ Then $S_{\frac{k}{2}}(\Gamma_{0}(4N),F)$ is generated by $f_{1}$ and $f_{2}$ if $\alpha_{2}\neq\alpha'_{2}$, and it is generated by $f_{1}$ and $f'_{2}$ if $\alpha_{2}\neq\alpha'_{2}$. \end{theorem}

In order to use Theorem $\ref{main}$ as an effective tool for computing a basis for $S_{\frac{3}{2}}(\Gamma_{0}(4N),F)$, we need to make the following observations. Once again, we assume that $\alpha_{2}\neq\alpha'_{2}$ to simplify the exposition and avoid unrewarding details. 

Let $h=2^{\frac{1}{2}}(f_{1}+f_{2})$ and write $\displaystyle{h=\sum_{n\geq1}a_{n}(h)q^{n}}$. First, we compute the Fourier coefficients $a_{n}(h)$ for positive square-free integers $n$ such that $\left(\frac{n}{p}\right)=-p^{\frac{1}{2}}\lambda_{p}'$ whenever $p|N$ and $p\nmid n$ $($since, for square-free $n$, $a_{n}(h)=0$ otherwise$)$. A straightforward calculation using Theorem $3.7$, $($\ref{K1}$)$, $($\ref{K2}$)$, $($\ref{K3}$)$, $($\ref{cp0}$)$, $($\ref{cps}$)$, $($\ref{c2}$)$, and $($\ref{c2'}$)$ shows that $$a_{n}(h)=\begin{cases} a_{n}(g)(b_{2}-2(2,n)_{2})& \text{if } n\equiv3\mod4;\\ 2a_{4n}(g) & \text{otherwise.}\end{cases}$$

Let us now consider a positive integer $n$, which is not divisible by any square prime to $4N$, and write $n=n^{\mathrm{sf}}y^{2}$. We also assume that $\left(\frac{n^{\mathrm{sf}}}{p}\right)=-p^{\frac{1}{2}}\lambda_{p}'$ whenever $p|N$ and $p\nmid n^{\mathrm{sf}}$. To simplify notation, we set $s=v_{2}(y)$. Another simple calculation shows that $$a_{n}(h)=\begin{cases} 2^{\frac{s+1}{2}}a_{n^{\mathrm{sf}}}(g)(\alpha_{2}^{s+1}+\alpha_{2}'^{s+1}-2^{-\frac{1}{2}}(2,n^{\mathrm{sf}})_{2}(\alpha_{2}^{s}+\alpha_{2}'^{s}))\prod_{p\mid N}b_{p}^{v_{p}(y)}& \text{if } n^{\mathrm{sf}}\equiv3\mod4;\\ 2^{\frac{s}{2}-1}a_{n^{\mathrm{sf}}}(h)(\alpha_{2}^{s}+\alpha_{2}'^{s})\prod_{p\mid N}b_{p}^{v_{p}(y)} & \text{otherwise.}\end{cases}$$

Finally, we use Proposition 14 on page 210 in \cite{koblitz} to compute $a_{n}(h)$ for an arbitrary positive integer $n$. Suppose that $n_{0}$ is a positive integer, which is not divisible by the square of any prime $p\nmid4N$. Then this proposition provides us with the following recursive formulae which hold for all $p\nmid 4N$ and $n_{1}$ prime to $p$: $$a_{n_{0}n_{1}^{2}p^{2}}(h)=a_{n_{0}n_{1}^{2}}(h)\left(b_{p}-\left(\frac{-n_{0}}{p}\right)\right)$$ and $$a_{n_{0}n_{1}^{2}p^{2(v+1)}}(h)=a_{n_{0}n_{1}^{2}p^{2v}}(h)\left(b_{p}-\left(\frac{-n_{0}}{p}\right)\right)-pa_{n_{0}n_{1}^{2}p^{2(v-1)}}(h)\text{ }\text{ }\text{ }v=1,2, ...$$

We implemented all of the above formulae in Sage. The result is a function that outputs the Fourier expansion of $h$ up to a desired precision. Thus, we get a basis for $S_{\frac{3}{2}}(\Gamma_{0}(4N),F)$ consisting of the modular forms $g$ and $h$.
\section{Examples}

We are interested in computing a basis for the space $S_{\frac{3}{2}}(\Gamma_{0}(60),F)$, where $F$ is the newform in $S_{2}^{\mathrm{new}}(\Gamma_{0}(15))$ corresponding to the elliptic curve $y^2 + xy + y = x^3 + x^2 - 10x - 10$. We know that 
 $$F =\sum_{n\geq1}b_{n}q^{n}=q - q^{2} - q^{3} - q^{4} + q^{5} + q^{6} + 3q^{8} + q^{9} - q^{10} - 4q^{11} + O(q^{12}).$$
Since $W_{3}(F)=F$ and $W_{5}(F)=-F$, we consider the definite quaternion algebra $B$ ramified at $5$. Let $O$ be an order in $B$ of level $N=15$. $O$ has two left ideal classes $[I_{1}]$ and $[I_{2}]$ which give rise to two orders $O_{1}=O_{r}(I_{1})$ and $O_{2}=O_{r}(I_{2})$.  The subgroups $R_{1}$ and $R_{2}$ of trace zero elements in $\mathbb{Z}+2O_{1}$ and $\mathbb{Z}+2O_{2}$ have bases $\{ i + \frac{1}{3}j + \frac{4}{3}k, \frac{2}{3}j + \frac{8}{3}k, 3k \}$ and $\{ \frac{4}{3}j - \frac{2}{3}k, i + \frac{2}{3}j - \frac{1}{3}k,  j - 2k \}$ respectively. The two theta series generated by the norm form of the quaternion algebra evaluated on $R_{1}$ and $R_{2}$ are:
 \begin{align*}g(I_{1})=\sum_{x\in R_{1}}q^{N(x)}&=\frac{1}{2}+ q^3 + q^{12} + 3q^{20} + 6q^{23} + q^{27} + 6q^{32} + 6q^{47} + q^{48} + 3q^{60}\\& +6q^{63}+6q^{68}+6q^{72}+q^{75}+3q^{80}+6q^{83}+6q^{87}+6q^{92}+6q^{95}+O(q^{100})\end{align*} and \begin{align*}g(I_{2})=\sum_{x\in R_{2}}q^{N(x)}&=\frac{1}{2} + 2q^8 + q^{12} + q^{15} + q^{20} + 4q^{23} + 2q^{27} + 4q^{32} + 2q^{35} + 8q^{47} + 3q^{48} + q^{60}\\& +6q^{63}+6q^{68}+4q^{72}+5q^{80}+4q^{83}+4q^{87}+10q^{92}+6q^{95}+ O(q^{100}).\end{align*}
The vector $I_{F}=I_{1}-I_{2}$ is in the $F$-isotypical component of $X(\mathbb{R})$. Recall that $I_{F}$ is unique up to a scalar multiple and satisfies $t_{p}(I_{F})=b_{p}I_{F}$ for all $p\nmid N$, where the action of the Hecke operators $\{t_{p}\}_{p\nmid N}$ is determined by the Brandt matrices $\{B_{p}\}_{p\nmid N}$. Hence, one can find the eigenvector $I_{F}$ using the Brandt module package in Sage by computing a few Brandt matrices and knowing a few Fourier coefficients of $F$. 

Hence, the modular form \begin{align*}g=g(I_{1})-g(I_{2})&=q^3 - 2q^8 - q^{15} + 2q^{20} + 2q^{23} - q^{27} + 2q^{32} - 2q^{35} - 2q^{47} - 2q^{48} + 2q^{60}\\& +2q^{72}+q^{75}-2q^{80}+2q^{83}+2q^{87}-4q^{92}+ O(q^{100})\end{align*}
 lies in the space $S_{\frac{3}{2}}(\Gamma_{0}(60),F)$. 
 
We shall now apply the theorem of Waldspurger to compute a basis for this space. In order to determine the local objects necessary to calculate the local factor $\prod{c_{p}(n)}$, we use the explicit formulae given in \cite{waldspurger}. As a result of these calculations, we get $N(F)=2^{2}\times3\times5$ and $S_{\frac{3}{2}}(\Gamma_{0}(60),F)=U(60,F,A_{F})$. We also get: $$U_{3}(1,F)=\{c_{3}^{s}[\lambda_{3}]\}\text{ }\text{ }\text{ }( \lambda_{3}=3^{-\frac{1}{2}}b_{3}),$$ $$U_{5}(1,F)=\{c_{5}^{s}[\lambda_{5}]\}\text{ }\text{ }\text{ }( \lambda_{5}=5^{-\frac{1}{2}}b_{5}),$$ $$U_{2}(2,F)=\{c_{2}'[\alpha_{2}], c_{2}'[\alpha'_{2}]\}\text{ }\text{ }\text{ }(\alpha_{2}+\alpha'_{2}=2^{-\frac{1}{2}}b_{2},\text{ }\alpha_{2}\alpha'_{2}=1),$$ $$U_{p}(0,F)=\{c_{p}^{0}[\lambda_{p}]\},\text{ for all }p\neq2,3,5\text{ }\text{ }\text{ }(\lambda_{p}=p^{-\frac{1}{2}}b_{p}).$$ Since the set $U(60,F)$ constitutes of two elements $\mathbf{c_{1}}$ and $\mathbf{c_{2}}$, the space $U(60,F,A_{F})$ is generated by the two functions $f(\mathbf{c_{1}},A_{F})$ and $f(\mathbf{c_{2}},A_{F})$ whose Fourier coefficients are calculated as follows: $$a_{n}(\mathbf{c_{1}},A_{F})=A_{F}(n^{\mathrm{sf}})n^{\frac{1}{4}}\underbrace{c_{3}^{s}[\lambda_{3}](n)c_{5}^{s}[\lambda_{5}](n)c'_{2}[\alpha_{2}](n)\prod_{p\neq 2,3,5}c_{p}^{0}[\lambda_{p}](n)}_{K_{1}(n)}$$
$$a_{n}(\mathbf{c_{2}},A_{F})=A_{F}(n^{\mathrm{sf}})n^{\frac{1}{4}}\underbrace{c_{3}^{s}[\lambda_{3}](n)c_{5}^{s}[\lambda_{5}](n)c'_{2}[\alpha'_{2}](n)\prod_{p\neq 2,3,5}c_{p}^{0}[\lambda_{p}](n)}_{K_{2}(n)}.$$
We carried out these computations and obtained the first few local factors $K_{1}$ and $K_{2}$.\\
\begin{table}[ht]
\caption{The Local Factors $K_{1}$ and $K_{2}$}
\centering
\begin{tabular}{|c|cc|c|cc|}
\hline
$n$ & $K_{1}(n)$ & $K_{2}(n)$ & $n$ & $K_{1}(n)$ & $K_{2}(n)$\\
\hline
2 &  2 &  2  & 13 & 0 & 0\\
3 & $\frac{1+i\sqrt{7}}{2}$ & $ \frac{1-i\sqrt{7}}{2}$ & 14 & 0 & 0\\
4 & 0 & 0 & 15 & $\frac{-3+i\sqrt{7}}{2\sqrt{2}}$ & $ \frac{-3-i\sqrt{7}}{2\sqrt{2}}$\\
5 & $\sqrt{2}$ & $\sqrt{2}$ & 16 & 0 & 0\\
6 & 0 & 0 & 17 & 2 & 2\\
7 & 0 & 0 &18 & $-\frac{2}{\sqrt{3}}$ & $-\frac{2}{\sqrt{3}}$  \\
8 & $\frac{-1+i\sqrt{7}}{\sqrt{2}}$ & $\frac{-1-i\sqrt{7}}{\sqrt{2}}$ & 19 & 0 & 0\\
9 & 0 & 0 & 20 & $\frac{-1+i\sqrt{7}}{2}$ & $ \frac{-1-i\sqrt{7}}{2}$\\
10 & 0 & 0 & 21 & 0 & 0 \\
11 & 0 & 0 &22 & 0 & 0 \\
12 & $-\sqrt{2}$ & $-\sqrt{2}$ &23 & $\frac{-3+i\sqrt{7}}{\sqrt{2}}$ & $ \frac{-3-i\sqrt{7}}{\sqrt{2}}$\\
\hline
\end{tabular}
\end{table}\\
Therefore,\begin{align*}f(\mathbf{c_{1}},A_{F}) &=A_{F}(2)2^{\frac{1}{4}}(2)q^{2}+A_{F}(3)3^{\frac{1}{4}}\left(\frac{1+i\sqrt{7}}{2}\right)q^{3}+A_{F}(5)5^{\frac{1}{4}}(\sqrt{2})q^{5}\\&+A_{F}(2)2^{\frac{1}{4}}(-1+i\sqrt{7})q^{8}+A_{F}(3)3^{\frac{1}{4}}(-2)q^{12} +A_{F}(15)15^{\frac{1}{4}}\left(\frac{-3+i\sqrt{7}}{2\sqrt{2}}\right)q^{15}\\&+A_{F}(17)17^{\frac{1}{4}}(2)q^{17}+A_{F}(2)2^{\frac{1}{4}}(-2)q^{18}+A_{F}(5)5^{\frac{1}{4}}\left(\frac{-1+i\sqrt{7}}{\sqrt{2}}\right)q^{20}\\&+A_{F}(23)23^{\frac{1}{4}}\left(\frac{-3+i\sqrt{7}}{\sqrt{2}}\right)q^{23}+O(q^{24}),\end{align*}and 
\begin{align*}f(\mathbf{c_{2}},A_{F})&=A_{F}(2)2^{\frac{1}{4}}(2)q^{2}+A_{F}(3)3^{\frac{1}{4}}\left(\frac{1-i\sqrt{7}}{2}\right)q^{3}+A_{F}(5)5^{\frac{1}{4}}(\sqrt{2})q^{5}\\&+A_{F}(2)2^{\frac{1}{4}}(-1-i\sqrt{7})q^{8}+A_{F}(3)3^{\frac{1}{4}}(-2)q^{12}+A_{F}(15)15^{\frac{1}{4}}\left(\frac{-3-i\sqrt{7}}{2\sqrt{2}}\right)q^{15}\\&+A_{F}(17)17^{\frac{1}{4}}(2)q^{17}+A_{F}(2)2^{\frac{1}{4}}(-2)q^{18}+A_{F}(5)5^{\frac{1}{4}}\left(\frac{-1-i\sqrt{7}}{\sqrt{2}}\right)q^{20}\\&+A_{F}(23)23^{\frac{1}{4}}\left(\frac{-3-i\sqrt{7}}{\sqrt{2}}\right)q^{23}+O(q^{24}).\end{align*} 
Using Theorem \ref{prop} to calculate the global factors $A_{F}(t)$ gives the following values:

\begin{table}[ht]
\caption{The Global Factors $A_{F}(t)$}
\centering
\begin{tabular}{|c|c|}
\hline
$t$ & $A_{F}(t)$\\
\hline
2 &  $-r8^{-\frac{1}{4}}$\\
3 & $r2^{-\frac{1}{2}}3^{-\frac{1}{4}}$ \\
5 & $r5^{-\frac{1}{4}}$\\
15 & $-r15^{-\frac{1}{4}}$\\
17 & 0   \\
23 & $r23^{-\frac{1}{4}}$\\
\hline
\end{tabular}
\end{table}

Notice that the global factor $A_{F}(17)$ is zero because the coefficient $c_{68}$ of $g$ is zero. We apply the following linear combinations:
\begin{align*}f&=\frac{\sqrt{2}}{i\sqrt{7}}(f(\mathbf{c_{2}},A_{F})-f(\mathbf{c_{1}},A_{F}))=2^{\frac{1}{2}}A_{F}(3)3^{\frac{1}{4}}q^{3}+2A_{F}(2)8^{\frac{1}{4}}q^{8}+A_{F}(15)15^{\frac{1}{4}}q^{15}\\&+2A_{F}(5)5^{\frac{1}{4}}q^{20}+2A_{F}(23)23^{\frac{1}{4}}q^{23}+O(q^{24})\end{align*}and 
\begin{align*}h&= \sqrt{2}(f(\mathbf{c_{1}},A_{F})+f(\mathbf{c_{2}},A_{F}))\\&=4A_{F}(2)2^{\frac{3}{4}}q^{2}+A_{F}(3)2^{\frac{1}{2}}3^{\frac{1}{4}}q^{3} +4A_{F}(5)5^{\frac{1}{4}}q^{5}-2A_{F}(2)2^{\frac{3}{4}}q^{8}-4A_{F}(3)2^{\frac{1}{2}}3^{\frac{1}{4}}q^{12}\\&-3A_{F}(15)15^{\frac{1}{4}}q^{15}+4A_{F}(17)2^{\frac{1}{2}}17^{\frac{1}{4}}q^{17}-4A_{F}(2)2^{\frac{3}{4}}q^{18}-2A_{F}(5)5^{\frac{1}{4}}q^{20}\\&-6A_{F}(23)23^{\frac{1}{4}}q^{23} +O(q^{24}).\end{align*} After substituting the values of $A_{F}(t)$ obtained above, we get $$f=g=q^3 - 2q^8 - q^{15} + 2q^{20} + 2q^{23}+O(q^{24})$$ and $$h=-4q^{2}+q^{3}+4q^{5}+2q^{8}-4q^{12}+3q^{15}+4q^{18}-2q^{20}-6q^{23}+O(q^{24}).$$ 
Therefore, the modular forms $g$ and $h$ form a basis for the space $S_{\frac{3}{2}}(\Gamma_{0}(60),F)$.

In what follows, we generate more examples of the modular forms $g$ and $h$ which form a basis for the space $S_{\frac{3}{2}}(\Gamma_{0}(4N),F)$, using the functions we created in Sage to implement calculations of $h$ $($as described following Theorem $3.7$$)$ and $g$ $($as described in Section $1$$)$. These functions are available from the author upon request. The open source mathematical software package may be found at $\mathbf{http://www.sagemath.org/}$.

First, we compute a basis $\{g1,h1\}$ for the space $S_{\frac{3}{2}}(\Gamma_{0}(44),F1)$ with $F1$ being the newform in $S_{2}^{\mathrm{new}}(\Gamma_{0}(11))$ corresponding to the elliptic curve $y^{2}+y=x^{3}-x^{2}-10x-20$.

\begin{sageexample}
sage: F1=Newforms(11)[0]
sage: g1=shimura_lift_in_kohnen_subspace(F1,11,40)
sage: g1
sage: h1=shimura_lift_not_in_kohnen_subspace(F1,11,40)
sage: h1
\end{sageexample}

In the second example, we compute a basis $\{g2,h2\}$ for the space $S_{\frac{3}{2}}(\Gamma_{0}(84),F2)$ with $F2$ being the newform in $S_{2}^{\mathrm{new}}(\Gamma_{0}(21))$ corresponding to the elliptic curve $y^{2}+xy=x^{3}-4x-1$.

\begin{sageexample}
sage: F2=Newforms(21)[0]
sage: g2=shimura_lift_in_kohnen_subspace(F2,21,40)
sage: g2
sage: h2=shimura_lift_not_in_kohnen_subspace(F2,21,40)
sage: h2
\end{sageexample}

In the third example, we compute a basis $\{g3,h3\}$ for the space $S_{\frac{3}{2}}(\Gamma_{0}(132),F3)$ with $F3$ being the newform in $S_{2}^{\mathrm{new}}(\Gamma_{0}(33))$ corresponding to the elliptic curve $y^{2}+xy=x^{3}+x^{2}-11x$.

\begin{sageexample}
sage: F3=Newforms(33)[0]
sage: g3=shimura_lift_in_kohnen_subspace(F3,33,40)
sage: g3
sage: h3=shimura_lift_not_in_kohnen_subspace(F3,33,40)
sage: h3
\end{sageexample}

Next, we compute a basis $\{g4,h4\}$ for the space $S_{\frac{3}{2}}(\Gamma_{0}(140),F4)$ with $F4$ being the newform in $S_{2}^{\mathrm{new}}(\Gamma_{0}(35))$ corresponding to the elliptic curve $y^{2}+y=x^{3}+x^{2}+9x+1$.

\begin{sageexample}
sage: F4=Newforms(35,names='a')[0]
sage: g4=shimura_lift_in_kohnen_subspace(F4,35,40)
sage: g4
sage: h4=shimura_lift_not_in_kohnen_subspace(F4,35,40)
sage: h4
\end{sageexample}

In the following example, we compute a basis $\{g5,h5\}$ for the space $S_{\frac{3}{2}}(\Gamma_{0}(148),F5)$ with $F5$ being the newform in $S_{2}^{\mathrm{new}}(\Gamma_{0}(37))$ corresponding to the elliptic curve $y^{2}+y=x^{3}+x^{2}-23x-50$.

\begin{sageexample}
sage: F5=Newforms(37)[1]
sage: g5=shimura_lift_in_kohnen_subspace(F5,37,40)
sage: g5
sage: h5=shimura_lift_not_in_kohnen_subspace(F5,37,40)
sage: h5
\end{sageexample}

Finally, we compute a basis $\{g6,h6\}$ for the space $S_{\frac{3}{2}}(\Gamma_{0}(924),F6)$ with $F6$ being the newform in $S_{2}^{\mathrm{new}}(\Gamma_{0}(231))$ corresponding to the elliptic curve $y^{2}+xy+y=x^{3}+x^{2}-34x+62$.

\begin{sageexample}
sage: F6=Newforms(231,names='a')[0]
sage: g6=shimura_lift_in_kohnen_subspace(F6,231,40)
sage: g6
sage: h6=shimura_lift_not_in_kohnen_subspace(F6,231,40)
sage: h6
\end{sageexample}

\bibliographystyle{siam}
\bibliography{myrefs}

\end{document}